\documentclass{article}
\usepackage{amsmath, amssymb, amsthm}
\usepackage[all]{xy}
\theoremstyle{remark}
\newtheorem{theorem}{\bf Theorem}[section]
\newtheorem{proposition}[theorem]{\bf Proposition}
\newtheorem{lemma}[theorem]{\bf Lemma}
\newtheorem{definition}[theorem]{\bf Definition}
\newtheorem{remark}[theorem]{\bf Remark}
\newtheorem{corollary}[theorem]{\bf Corollary}
\newtheorem{case}{\bf Case}
\newtheorem*{rem}{Remark}
\author{Kin Wai (Edisy) Chan \\ Mentor: Dr. Eric Wambach}
\title{Zeta-functions of Curves over Finite Fields}
\date{}

\begin{document}
\maketitle

\begin{abstract}
Curves over finite fields are of great importance in cryptography and coding theory.  Through studying their zeta-functions, we would be able to find out vital arithmetic and geometric information about them and their Jacobians, including the number of rational points on this kind of curves.  In this paper, I investigate if it is possible to construct a curve over finite fields of a given genus $g$ whose zeta-function is given as a product of zeta-functions of $g$ elliptic curves, and find out alternative methods if it is not possible.  Basically, I look for conditions which those $g$ elliptic curves should satisfy such that their product (of their Jacobians) is isogenous to the Jacobian of a curve of a given genus $g$.  Then from this isogenous relationship I can determine the characteristic polynomial of the Frobenius endomorphism of the Jacobian of the new curve and by this characteristic polynomial I can thus determine the zeta-function of this new curve.  By using the zeta-functions of curves in the form as generating functions, the number of rational points on curves can even be found out, which may lead to further researches relating to some applications in cryptography, coding theory and even information theory.
\end{abstract}

\section{Introduction}

\quad It is known that to determine all possible zeta-functions of the curves over finite fields of a given genus g is a very difficult task. In some sense, this is the same task of determining whether a given abelian variety is isogenous to a Jacabian.  For $g=1$, the problem is completely solved by Deuring\cite{bib: zeta3}. For g=2, Serre\cite{bib: zeta4} and R$\mathrm{\ddot{u}}$ck\cite{bib: zeta5} give the partial answer.

In this paper, I investigate for which $g$ it is possible to apply an elementary construction method using elliptic curves to find some of the zeta-functions of curves of genus $g$ whose Jacobian is isogenous to the product of $g$ elliptic curves. That is to say, I find out when and how the product of $g$ elliptic curves satisfying certain conditions is isogenous to the Jacobian of a curve of genus $g$. This is possible for $g=3$ (Corollary 3.4), but it is impossible for $g\geq 4$ (Remark 4.3). A direct result of Corollary 3.4 is Theorem 4.2 where the conditions are more numerical than those in Corollary 3.4.

There is a fact that the reciprocal of the numerator of the zeta-function of a curve is the characteristic polynomial of the Frobenius endomorphism of the Jacobian of the curve. Hence, it suffices to determine which polynomial can become the characteristic polynomial of the Frobenius endomorphism of the Jacobian of a curve of genus 3 in order to determine some of the zeta-functions of curves of genus 3.

In Section 2, I introduce some basic facts about elliptic curves and study 2-torsion points. An elementary construction method using elliptic curves is presented in Section 3. I determine part of the characteristic polynomials of the Frobenius endomorphisms of the Jacobians of curves of genus 3 and explain briefly why this cannot be done for $g\geq 4$ in Section 4.

I adopt the following notations throughout the paper:

\begin{itemize}
\item $q=p^r$: $p$ is a prime and $r\in\mathbb{N}$;
\item $\chi/\mathbb{F}_q$: a (smooth projective) algebraic curve over $\mathbb{F}_q$;
\item $\mathcal{E}/\mathbb{F}_q$: an elliptic curve over $\mathbb{F}_q$;
\item $g=g(\chi)$: the genus of $\chi$;
\item $\mathbb{F}_q(\chi)$: the function field of $\chi$;
\item $\chi(\mathbb{F}_q)$: the set of the $\mathbb{F}_q$--rational points of $\chi$;
\item $\mathcal{J}_\chi$: the Jacobian variety of $\chi$;
\item $\pi_\mathcal{A}$: the Frobenius morphism of a variety $\mathcal{A}/\mathbb{F}_q$;
\item $\pi$: the Frobenius morphism of $\mathbb{P}^1/\mathbb{F}_q$.
\end{itemize}

\section{Elliptic curve}
\quad An elliptic curve is a 1--dimensional smooth projective variety of genus 1 with a specified base point, i.e. an algebraic curve of genus 1 having a specified rational point.  By considering the map $\mathcal{E}/\mathbb{F}_q\rightarrow\mathbb{P}^1/\mathbb{F}_q$  such that $[x: y: 1]\mapsto[x: 1]$ and $[0:1:0]\mapsto[1:0]$ , we know that $\mathcal{E}$  is a covering of degree 2 over the projective line  $\mathbb{P}^1$ with set of ramified points being the set $\mathcal{E}[2]$ of the 2-torsion points.  It is clear that $|\mathcal{E}[2]|=4$ for odd $p$ and $|\mathcal{E}[2]|=1$ or $2$ for $p=2$.

Since I attempt to apply a similar method in \cite{bib: zeta1} to do the generalization, I need to use the following lemma from \cite{bib: zeta9} as well.

\bigskip
\begin{lemma}
\cite[Theorem 4.1]{bib: zeta9} Let $q=p^r$ a prime $p$ and an integer $r\geq 1$. Let $t$ be an integer with $|t|\leq 2\sqrt{q}$. Then there is an elliptic curve over $\mathbb{F}_q$ with $q+1+t$ rational points if and only if $t$ satisfies one of the following conditions:
\newcounter{count0}
\begin{list}{(\roman{count0})}{\usecounter{count0}}
\item $\gcd(p,t)=1$;
\item $r$ is even: $t=\pm 2\sqrt{q}$;
\item $r$ is even and $p\not\equiv 1\pmod 3$: $t=\pm\sqrt{q}$;
\item $r$ is odd and $p=2$ or 3: $t=\pm\sqrt{pq}$;
\item $r$ is odd: $t=0$;
\item $r$ is even and $p\not\equiv 1\pmod 4$: $t=0$.
\end{list}
\end{lemma}

\bigskip
Let $R_\mathcal{E}$ be the set of ramified points of  $\mathbb{P}^1$ for the covering $\mathcal{E}/\mathbb{F}_q\rightarrow\mathbb{P}^1/\mathbb{F}_q$. By looking at the form of the covering map, we can easily know that $|R_\mathcal{E}|=|\mathcal{E}[2]|$. Conversely, if we have $R\subseteq\mathbb{P}^1$ such that $\pi(R)=R$ and a $Q\in R$ such that $\pi(Q)=Q$ for the Frobenius map $\pi$, and if $|R|=4$ for odd $p$ and $|R|=1$ or 2 for $p=2$ as well, then there exists an automorphism $g$ in the projective general linear group $\mathrm{PGL}(2, q)$ such that $g(Q)=[1: 0]$ and an elliptic curve $\mathcal{E}/\mathbb{F}_q$ such that $R_\mathcal{E}=R$.

\begin{displaymath}
\xymatrix
{\mathcal{E}[2]\ar[dd]&&\mathcal{E}\ar[dd] \\ \\
R_\mathcal{E}&&\mathbb{P}^1}
\end{displaymath}

\bigskip
Now if we suppose that $p$ is odd, then an elliptic curve $\mathcal{E}$ is totally determined by $R_\mathcal{E}$. That is to say if there exists an automorphism $\sigma\in \mathrm{PGL}(2,q)$ such that $\sigma(R_{\mathcal{E}_1})=R_{\mathcal{E}_2}$, then $\mathcal{E}_1\cong \mathcal{E}_2$. Throughout this paper, p is assumed to be odd, so we have $|R_\mathcal{E}|=|\mathcal{E}[2]|=4$ for any elliptic curve $\mathcal{E}/\mathbb{F}_q$.

From now on, I fix the map $\mathcal{E}/\mathbb{F}_q\rightarrow\mathbb{P}^1/\mathbb{F}_q$ so that the notation $R_\mathcal{E}$ makes sense.

In studying elliptic curves, we sometimes use another more convenient form of Weierstrass equation for an elliptic curve.  This is called the \textit{Legendre form}.

\bigskip
\begin{definition}
An elliptic curve $\mathcal{E}$ is in \textit{Legendre form} if its Weierstrass equation can be written as
\[
y^2=x(x-1)(x-\lambda).
\]
\end{definition}

\bigskip
In this paper, $\lambda$ is called a \textit{Legendre coefficient of $\mathcal{E}$}. If an elliptic curve $\mathcal{E}$ is in Legendre form, then $\mathcal{E}$ is denoted as $\mathcal{E}_\lambda$ instead. In fact, we have the following proposition illustrating the usefulness of Legendre form.

\bigskip
\begin{proposition}
\cite[Proposition III.1.7 (a)(b)]{bib: zeta6} Let $p$ be an odd prime. Then every elliptic curve $\mathcal{E}/\mathbb{F}_q$ is isomorphic (over $\bar{\mathbb{F}}_q$) to an elliptic curve in \textit{Legendre form}
\[
\mathcal{E}_\lambda:y^2=x(x-1)(x-\lambda)
\]
for some $\lambda\in \bar{\mathbb{F}}_q$, $\lambda\neq 0,1$ (Note that $\mathcal{E}_\lambda$ is \textit{defined over $\bar{\mathbb{F}}_q$}). Moreover, the \textit{$j$-invariant of $\mathcal{E}_\lambda$} is
\[
j(\mathcal{E}_\lambda)=2^8\frac{(\lambda^2-\lambda+1)^3}{\lambda^2(\lambda-1)^2}.
\]
\end{proposition}

\bigskip
In order to facilitate our discussion on elliptic curves, I define a new term regarding the relation between $\lambda_i$ of elliptic curves $\mathcal{E}_{\lambda_i}$.

\bigskip
\begin{definition}
Let $\lambda_1$ and $\lambda_2$ be Legendre coefficients of $\mathcal{E}_{\lambda_1}$ and $\mathcal{E}_{\lambda_2}$ respectively.  Then $\lambda_1$ is \textit{Legendre equivalent} to $\lambda_2$ if
\[
\lambda_2\in\{\lambda_1,1-\lambda_1,\frac{1}{\lambda_1},\frac{1}{1-\lambda_1},\frac{\lambda_1-1}{\lambda_1},\frac{\lambda_1}{\lambda_1-1}\}.
\]
\end{definition}

\bigskip
It is not difficult to check that it is indeed an equivalence relation.  Then we have an important proposition telling us when two elliptic curves in Legendre forms is isomorphic to each other.

\bigskip
\begin{proposition}
\cite[Proposition 4(1.3)]{bib: zeta7} Let $\mathcal{E}_{\lambda_1}$ and $\mathcal{E}_{\lambda_2}$ be two elliptic curves in Legendre forms. Then $\mathcal{E}_{\lambda_1}$ is isomorphic to $\mathcal{E}_{\lambda_2}$ (over $\bar{\mathbb{F}}_q$)\footnote{In fact, to be more precise, $\mathcal{E}_{\lambda_1}$ is isomorphic to $\mathcal{E}_{\lambda_2}$ over $\mathbb{F}_q(j(\mathcal{E}_{\lambda_1}))=\mathbb{F}_q(j(\mathcal{E}_{\lambda_2}))$.} if and only if $\lambda_1$ is Legendre equivalent to $\lambda_2$.
\end{proposition}

\bigskip
The following lemma is also used in an essential place in the Lemma 2.7, which is a significant lemma in this paper.

\bigskip
\begin{lemma}
\cite[Lemma 2.3]{bib: zeta1} Let $\mathcal{E}_i/\mathbb{F}_q$ $(i=1,2)$ be elliptic curves such that $\mathcal{E}_1$ is not isomorphic to $\mathcal{E}_2$ and $|\mathcal{E}_1(\mathbb{F}_q)\cap\mathcal{E}_1[2]|=|\mathcal{E}_2(\mathbb{F}_q)\cap\mathcal{E}_2[2]|$, then there exists an elliptic curve $\mathcal{E}_3/\mathbb{F}_q$ such that $\mathcal{E}_3$ is isomorphic to $\mathcal{E}_2$ and $|R_{\mathcal{E}_1}\cap R_{\mathcal{E}_3}|=3$.
\end{lemma}

\bigskip
Now we have sufficient machinery to prove the Lemma 2.7, which will be used later in an important place. Note that $j_i$ represents the $j$-invariant of $\mathcal{E}_i$.

\bigskip
\begin{lemma}
Let $\mathcal{E}_i/\mathbb{F}_q$ $(i=1,2,3)$ be \textit{pairwise non-isomorphic} elliptic curves. Suppose $j_3\in \mathbb{F}_q$ and $\mathcal{E}_i/\mathbb{F}_q$ $(i=1,2,3)$ satisfy any one of the following sets of conditions:
\newcounter{count1}
\begin{list}{(\roman{count1})}{\usecounter{count1}}
\item $|\mathcal{E}_1(\mathbb{F}_q)\cap\mathcal{E}_1[2]|=|\mathcal{E}_2(\mathbb{F}_q)\cap\mathcal{E}_2[2]|=|\mathcal{E}_3(\mathbb{F}_q)\cap\mathcal{E}_3[2]|=2$;
\item $\displaystyle\frac{\lambda_2(1-\lambda_1)}{\lambda_2-\lambda_1}$ is Legendre equivalent to $\lambda_3$,
\end{list}
or
\newcounter{count2}
\begin{list}{(\roman{count2})}{\usecounter{count2}}
\item $|\mathcal{E}_1(\mathbb{F}_q)\cap\mathcal{E}_1[2]|=|\mathcal{E}_2(\mathbb{F}_q)\cap\mathcal{E}_2[2]|=|\mathcal{E}_3(\mathbb{F}_q)\cap\mathcal{E}_3[2]|=4$;
\item $\displaystyle\frac{\lambda_2}{\lambda_1}$ is Legendre equivalent to $\lambda_3$,
\end{list}
where $\lambda_1$, $\lambda_2$ and $\lambda_3$ are Legendre coefficients of $\mathcal{E}_1$, $\mathcal{E}_2$ and $\mathcal{E}_3$ respectively. Then there exists elliptic curves $\mathcal{E}'_i/\mathbb{F}_q$ $(i=1,2,3)$ such that $\mathcal{E}'_i$ is isomorphic to $\mathcal{E}_i$ over $\mathbb{F}_q$ ($\forall i=1,2,3$) and:
\newcounter{count3}
\begin{list}{(\roman{count3})}{\usecounter{count3}}
\item $|R_{\mathcal{E}'_i}\cap R_{\mathcal{E}'_j}|=3$ $\forall i\neq j$;
\item $|R_{\mathcal{E}'_1}\cap R_{\mathcal{E}'_2}\cap R_{\mathcal{E}'_3}|=2$.
\end{list}
\end{lemma}
\begin{proof}
By Lemma 2.6, we may assume without loss of generality that $\mathcal{E}_1$ and $\mathcal{E}_2$ possess the property $|R_{\mathcal{E}_1}\cap R_{\mathcal{E}_2}|=3$. We may also assume:
\[
\mathcal{E}_1:y^2=(x-p_1)(x-p_2)(x-p_3) \quad\mathrm{and}\quad \mathcal{E}_2:y^2=(x-p_1)(x-p_2)(x-p_4).
\]
Then we can consider $R_{\mathcal{E}_1}=\{\mathcal{O},[p_1:1],[p_2:1],[p_3:1]\}$ and $R_{\mathcal{E}_2}=\{\mathcal{O},[p_1:1],[p_2:1],[p_4:1]\}$, where $\mathcal{O}$ is the point of infinity. If we change $\mathcal{E}_1$ and $\mathcal{E}_2$ into Legendre forms, we may assume $\lambda_1=\displaystyle\frac{p_3-p_1}{p_2-p_1}$ and $\lambda_2=\displaystyle\frac{p_4-p_1}{p_2-p_1}$. Now, it is sufficient to find out an elliptic curve $\mathcal{E}'_3/\mathbb{F}_q$ which is isomorphic to $\mathcal{E}_3$ over $\mathbb{F}_q$ and satisfies the conditions:
\newcounter{count4}
\begin{list}{(\roman{count4})}{\usecounter{count4}}
\item $|R_{\mathcal{E}'_3}\cap R_{\mathcal{E}_j}|=3$ for $j=1,2$;
\item $|R_{\mathcal{E}_1}\cap R_{\mathcal{E}_2}\cap R_{\mathcal{E}'_3}|=2$.
\end{list}

Let $\phi_i$ be the covering $\mathcal{E}_i\rightarrow \mathbb{P}^1$.  Then $\phi_i$ maps $\mathcal{E}_i[2]$ onto $R_{\mathcal{E}_i}$ bijectively, and thus $\phi_i(\mathcal{E}_i(\mathbb{F}_q)\cap\mathcal{E}_i[2])=\mathbb{P}^1(\mathbb{F}_q)\cap R_{\mathcal{E}_i}$.  This implies $|\mathbb{P}^1(\mathbb{F}_q)\cap R_{\mathcal{E}_1}|=|\mathbb{P}^1(\mathbb{F}_q)\cap R_{\mathcal{E}_2}|=|\mathbb{P}^1(\mathbb{F}_q)\cap R_{\mathcal{E}_3}|$.

\bigskip
\begin{case}$|\mathbb{P}^1(\mathbb{F}_q)\cap R_{\mathcal{E}_i}|=2$

In this case, $R_{\mathcal{E}_i}$ contains two distinct rational points and two other conjugate points.  Without loss of generality, let $[p_1:1]$ and $[p_2:1]$ be the conjugate points.  Now, we consider $R_{\mathcal{E}'_3}=\{[p_1:1],[p_2:1],[p_3:1],[p_4:1]\}$. Then we have an elliptic curve $\mathcal{E}'_3$, which is defined over $\mathbb{F}_q$ (since $p_1$ and $p_2$ are conjugate pairs and $p_3,p_4\in \mathbb{F}_q$), with the desired conditions. It suffices to show that $\mathcal{E}'_3$ is isomorphic to $\mathcal{E}_3$ over $\mathbb{F}_q$.

Let $\sigma=
\bigg[\bigg(\begin{array}{cc}
0&1\\
1&-p_3
\end{array}\bigg)\bigg]
\in\mathrm{PGL}(2,q)$, then $\displaystyle\sigma(R_{\mathcal{E}'_3})=\{\mathcal{O},[\frac{1}{p_1-p_3}:1],[\frac{1}{p_2-p_3}:1],[\frac{1}{p_4-p_3}:1]\}$. The new elliptic curve $\sigma(\mathcal{E}'_3)$ determined by $\sigma(R_{\mathcal{E}'_3})$ is therefore isomorphic to $\mathcal{E}'_3$.  Since a Legendre coefficient of $\sigma(\mathcal{E}'_3)$ is
\begin{eqnarray*}
\lambda_\sigma&=&\frac{\frac{1}{p_4-p_3}-\frac{1}{p_1-p_3}}{\frac{1}{p_2-p_3}-\frac{1}{p_1-p_3}} \\
&=&\frac{\frac{p_1-p_4}{(p_4-p_3)(p_1-p_3)}}{\frac{p_1-p_2}{(p_2-p_3)(p_1-p_3)}} \\
&=&\frac{\frac{p_4-p_1}{p_2-p_1}(1-\frac{p_3-p_1}{p_2-p_1})}{\frac{p_4-p_1}{p_2-p_1}-\frac{p_3-p_1}{p_2-p_1}} \\
&=&\frac{\lambda_2(1-\lambda_1)}{\lambda_2-\lambda_1}
\end{eqnarray*}
which is Legendre equivalent to $\lambda_3$. By Proposition 2.5, this implies $\sigma(\mathcal{E}'_3)\cong \mathcal{E}_3$ (over $\mathbb{F}_q(j_3)$) and in turn $\mathcal{E}'_3\cong \mathcal{E}_3$ (over $\mathbb{F}_q(j_3)$). Since, however, $j_3\in \mathbb{F}_q$, $\mathcal{E}'_3$ is in fact isomorphic to $\mathcal{E}_3$ over $\mathbb{F}_q$.
\end{case}

\bigskip
\begin{case}$|\mathbb{P}^1(\mathbb{F}_q)\cap R_{\mathcal{E}_i}|=4$

In this case, all points in $R_{\mathcal{E}_i}$ are rational points. Now, we consider $R_{\mathcal{E}'_3}=\{\mathcal{O},[p_1:1],[p_3:1],[p_4:1]\}$. Again, we get an elliptic curve $\mathcal{E}'_3$, which is trivially defined over $\mathbb{F}_q$, satisfying the required conditions. Clearly, if we change $\mathcal{E}'_3$ into Legendre form, we may have $\displaystyle\lambda'_3=\frac{p_4-p_1}{p_3-p_1}=\frac{\frac{p_4-p_1}{p_2-p_1}}{\frac{p_3-p_1}{p_2-p_1}}=\frac{\lambda_2}{\lambda_1}$, which is Legendre equivalent to $\lambda_3$. We immediately get the desirable result $\mathcal{E}'_3\cong \mathcal{E}_3$ (over $\mathbb{F}_q$) similarly as in case 1.
\end{case}
\end{proof}

\section{Construction of a curve of genus 3}
\begin{proposition}
Let $p$ be an odd prime.  Let $\mathcal{E}_i/\mathbb{F}_q$ $(i=1,2,3)$ be \textit{pairwise non-isomorphic} elliptic curves and $R_{\mathcal{E}_i}$ $(i=1,2,3)$ be the sets of ramified points of $\mathbb{P}^1$ for $\mathcal{E}_i/\mathbb{F}_q\rightarrow\mathbb{P}^1/\mathbb{F}_q$ satisfy the following conditions:
\newcounter{count5}
\begin{list}{(\roman{count5})}{\usecounter{count5}}
\item $|R_{\mathcal{E}'_i}\cap R_{\mathcal{E}'_j}|=3$ $\forall i\neq j$;
\item $|R_{\mathcal{E}_1}\cap R_{\mathcal{E}_2}\cap R_{\mathcal{E}_3}|=2$;
\end{list}
Then there exists a curve $\chi/\mathbb{F}_q$ of genus 3 such that the Jacobian $\mathcal{J}_\chi$ is isogenous to $\mathcal{E}_1\times\mathcal{E}_2\times\mathcal{E}_3$.
\end{proposition}
\begin{proof}
$\mathbb{F}_q(\mathbb{P}^1)=\mathbb{F}_q(t)$ can be seen as the subfield of $\mathbb{F}_q(\mathcal{E}_i)$ $(i=1,2,3)$.  Clearly, for any $i\neq j$, $\mathbb{F}_q(\mathcal{E}_i)\cap\mathbb{F}_q(\mathcal{E}_j)=\mathbb{F}_q(\mathbb{P}^1)$ as $\mathbb{F}_q(\mathcal{E}_i) (i=1,2,3)$ is a quadratic extension of $\mathbb{F}_q(t)$.  Then we have a composite field $\mathbb{F}_q(\mathcal{E}_1)\cdot\mathbb{F}_q(\mathcal{E}_2)$ which is an abelian extension of $\mathbb{F}_q(t)$ of degree 4.  Now if $\mathbb{F}_q(\mathcal{E}_3)$ is not contained in $\mathbb{F}_q(\mathcal{E}_1)\cdot\mathbb{F}_q(\mathcal{E}_2)$, we will get a new composite field $\mathbb{F}_q(\mathcal{E}_1)\cdot\mathbb{F}_q(\mathcal{E}_2)\cdot\mathbb{F}_q(\mathcal{E}_3)$ which is an abelian extension of $\mathbb{F}_q(t)$ of degree 8.  Let $\mathcal{Y}/\mathbb{F}_q$ and $\chi/\mathbb{F}_q$ be (smooth) algebraic curves corresponding to the composite fields $\mathbb{F}_q(\mathcal{E}_1)\cdot\mathbb{F}_q(\mathcal{E}_2)$ and $\mathbb{F}_q(\mathcal{E}_1)\cdot\mathbb{F}_q(\mathcal{E}_2)\cdot\mathbb{F}_q(\mathcal{E}_3)$ respectively.

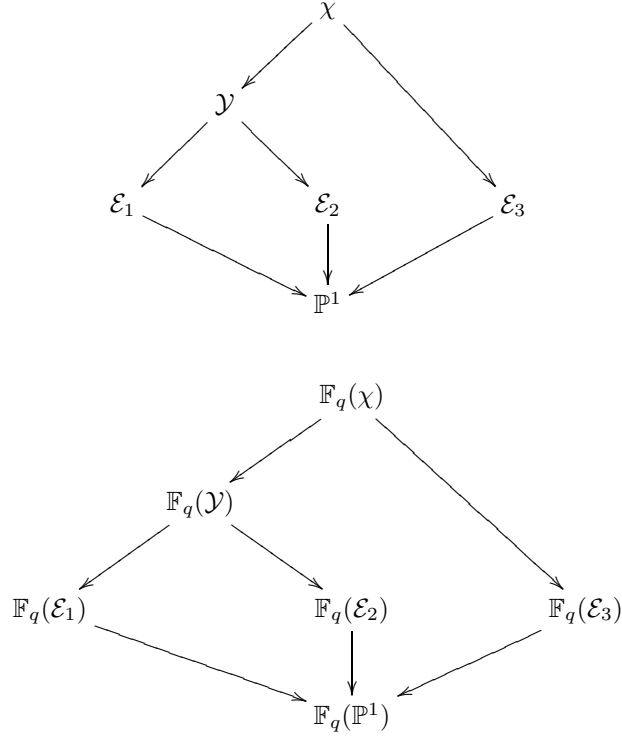
\begin{displaymath}
\xymatrix
{&&\chi\ar[ld]\ar[rrdd]&& \\
&\mathcal{Y}\ar[ld]\ar[rd]&&& \\
\mathcal{E}_1\ar[rrd]&&\mathcal{E}_2\ar[d]&&\mathcal{E}_3\ar[lld] \\
&&\mathbb{P}^1&&
}
\end{displaymath}
\bigskip
\begin{displaymath}
\xymatrix
{
&&\mathbb{F}_q(\chi)\ar[ld]\ar[rrdd]&& \\
&\mathbb{F}_q(\mathcal{Y})\ar[ld]\ar[rd]&&& \\
\mathbb{F}_q(\mathcal{E}_1)\ar[rrd]&&\mathbb{F}_q(\mathcal{E}_2)\ar[d]&&\mathbb{F}_q(\mathcal{E}_3)\ar[lld] \\
&&\mathbb{F}_q(\mathbb{P}^1)&&
}
\end{displaymath}

\bigskip
The set of ramified points of $\mathbb{P}^1$ for the covering $\chi\rightarrow\mathbb{P}^1$ is the union $R_{\mathcal{E}_1}\cup R_{\mathcal{E}_2}\cup R_{\mathcal{E}_3}$.  By Abhyankar's Lemma \cite[Proposition III.8.9]{bib: zeta8}, we know that every point of this set has ramification index 2.  In addition,
\begin{eqnarray*}
|R_{\mathcal{E}_1}\cup R_{\mathcal{E}_2}\cup R_{\mathcal{E}_3}|&=&|R_{\mathcal{E}_1}|+|R_{\mathcal{E}_2}|+|R_{\mathcal{E}_3}|-|R_{\mathcal{E}_1}\cap R_{\mathcal{E}_2}|-|R_{\mathcal{E}_2}\cap R_{\mathcal{E}_3}|-|R_{\mathcal{E}_3}\cap R_{\mathcal{E}_1}| \\
& & +|R_{\mathcal{E}_1}\cap R_{\mathcal{E}_2}\cap R_{\mathcal{E}_3}| \\
& = & 4+4+4-3-3-3+2\\
& = & 5.
\end{eqnarray*}
Therefore by the Hurwitz formula \cite[Theorem II.5.9]{bib: zeta6}, we get
\begin{eqnarray*}
2\textrm{genus}(\chi)-2& = &(2\textrm{genus}(\mathbb{P}^1)-2)(\textrm{deg}(\phi))+\sum_{P\in\chi}(e_\phi(P)-1) \\
& = & (2\times 0-2)\times 8+4\times 5\times (2-1) \\
& = & 4.
\end{eqnarray*}
(Note that each ramified point on $\mathbb{P}^1$ has four pre-images in $\chi$).  We thus have $\textrm{genus}(\chi)=3$.  Now it only remains to prove that $\mathbb{F}_q(\mathcal{E}_3)\not\leq\mathbb{F}_q(\mathcal{E}_1)\cdot\mathbb{F}_q(\mathcal{E}_2)$ and the Jacobian of $\chi$ is isogenous to $\mathcal{E}_1\times\mathcal{E}_2\times\mathcal{E}_3$.  The first statement can be verified by Lemma 3.2 and the second one can be proved by a special case of Theorem 3.3.
\end{proof}

\bigskip
\begin{lemma}
Let $p$ be an odd prime and $\mathcal{E}_i/\mathbb{F}_q$ $(i=1,2,3)$ be \textit{pairwise non-isomorphic} elliptic curves.  Then $\mathbb{F}_q(\mathcal{E}_3)\not\leq\mathbb{F}_q(\mathcal{E}_1)\cdot\mathbb{F}_q(\mathcal{E}_2)$. Here $\mathbb{F}_q(\mathcal{E}_i)$ $(i=1,2,3)$ is viewed as a fixed degree $2$ extension of $\mathbb{F}_q(\mathbb{P}^1)$.
\end{lemma}

\smallskip
\begin{rem}
An elementary proof of lemma 3.2 is included here for convenience. However it should be noted that this lemma indeed follows from theorem 3.3.
\end{rem}

\begin{proof}
Since $\mathrm{Char}(\mathbb{F}_q)\neq 2$, we can assume $\mathcal{E}_i/\mathbb{F}_q$ is given by the equation $y^2=f_i(x)$, where $f_i(x)$ is a cubic polynomial in $\mathbb{F}_q[x]$. Let's standardize $f_i$ by $f_i(x)=4x^3+b_{i,2}x^2+2b_{i,4}x+b_{i,6}$. Now let $K=\mathbb{F}_q(x)$. Then by considering $y^2-f_i(x)\in K[y]$, we may have
\begin{eqnarray*}
\mathbb{F}_q(\mathcal{E}_i)&=&\mathrm{Frac}(\mathbb{F}_q[x,y]/(y^2-f_i(x))) \\
&=&\mathbb{F}_q(x)(y_i) \\
&=&K(y_i),
\end{eqnarray*}
where $y_i$ is a root of the polynomial $y^2-f_i(x)$ in $K[y]$.  It suffices to prove $K(y_3)\not\leq K(y_1)\cdot K(y_2)$ instead. \\

Suppose the contrary is true, then $y_3$ is a linear combination of the basis $\{1,y_1,y_2,y_1y_2\}$ of $K(y_1)\cdot K(y_2)=K(y_1,y_2)$ over $K$.  Let $y_3=c_0+c_1y_1+c_2y_2+c_3y_1y_2$, where $c_j\in K$ $(j=0,1,2,3)$. Then
\begin{eqnarray}
y_3^2&=&c_0^2+2c_0c_1y_1+2c_0c_2y_2 \nonumber \\
& &+y_1^2(c_1^2+2c_1c_2y_2)+y_2^2(c_2^2+2c_2c_3y_1) \nonumber \\
& &+2c_0c_3y_1y_2+c_3^2y_1^2y_2^2 \nonumber \\
&=&C_0+f_1(x)C_1+f_2(x)C_2+C_3f_1(x)f_2(x),
\end{eqnarray}

\[
\mathrm{where}
\left\{
\begin{array}{ccc}
C_0&=&c_0^2+2c_0c_1y_1+2c_0c_2y_2+2c_0c_3y_1y_2 \\
C_1&=&c_1^2+2c_1c_2y_2 \\
C_2&=&c_2^2+2c_2c_3y_1 \\
C_3&=&c_3^2
\end{array}
\right.
\in K(y_1,y_2).
\]
However, since $y_3^2=f_3(x)$ which is a cubic polynomial in $\mathbb{F}_q[x]$ and $f_1(x)f_2(x)$ is a polynomial of degree 6 in $\mathbb{F}_q[x]$, by comparing coefficients in (1) we know that $c_3^2=C_3=0$. Therefore, $C_0+f_1(x)C_1+f_2(x)C_2=y_3^2\in K$.  It implies that $\forall \sigma\in \mathrm{Gal}(K(y_1,y_2)/K),$
\begin{eqnarray}
&\sigma(C_0+f_1(x)C_1+f_2(x)C_2)=C_0+f_1(x)C_1+f_2(x)C_2 \nonumber \\
\Rightarrow&(\sigma-\iota)C_0+f_1(x)(\sigma-\iota)C_1+f_2(x)(\sigma-\iota)C_2=0,
\end{eqnarray}
where $\iota\in \mathrm{Gal}(K(y_1,y_2)/K)$ is the identity.  Since \footnote{Note: $(\sigma-\iota)C_i$ is only a notation used for convenience to represent $\sigma C_i-C_i$}$(\sigma-\iota)C_2=2c_2c_3(\sigma-\iota)y_1=0$ for $c_3=0$, instead of (2) we should have
\begin{eqnarray}
(\sigma-\iota)C_0+f_1(x)(\sigma-\iota)C_1=0.
\end{eqnarray}

Now we know that
\[
\left\{
\begin{array}{ccc}
(\sigma-\iota)C_0&=&2c_0c_1(\sigma-\iota)y_1+2c_0c_2(\sigma-\iota)y_2 \\
(\sigma-\iota)C_1&=&2c_1c_2(\sigma-\iota)y_2
\end{array}
\right.
.
\]
Consider the automorphism $\sigma$ in $\mathrm{Gal}(K(y_1,y_2)/K)$ defined by
\[
\left\{
\begin{array}{ccccc}
y_1&\mapsto&\bar{y_1}&=& -y_1\\
y_2&\mapsto&\bar{y_2}&=& -y_2\\
y_1y_2&\mapsto&\bar{y_1}\bar{y_2}&=& y_1y_2
\end{array}
\right.
\]
(Since $y_i$ and $\bar{y_i}$ both satisfy the equation $y^2-f_i(x)=0$, $y_i\bar{y_i}=-f_i(x)=-y_i^2 \Rightarrow \bar{y_i}=-y_i$ as $y_i \neq 0$). Then we obtain
\[
\left\{
\begin{array}{ccc}
(\sigma-\iota)C_0&=&-4c_0c_1y_1-4c_0c_2y_2 \\
(\sigma-\iota)C_1&=&-4c_1c_2y_2
\end{array}
\right.
\]
which are \textit{linearly independent over $K$} if both $(\sigma-\iota)C_0$ and $(\sigma-\iota)C_1$ are non-vanishing. This is a contradiction to (3) and hence the lemma follows. If $(\sigma-\iota)C_0=0$ or $(\sigma-\iota)C_1=0$, then we get $f_1(x)(\sigma-\iota)C_1=0$ or $(\sigma-\iota)C_0=0$ respectively, which also implies $(\sigma-\iota)C_0=(\sigma-\iota)C_1=0$.  Then we can merely consider a few cases:
\newcounter{count9}
\begin{list}{(\roman{count9})}{\usecounter{count9}}
\item $c_0=0$, $c_1=0$ and $c_2\neq 0$;
\item $c_0=0$, $c_1\neq 0$ and $c_2=0$;
\item $c_0=0$, $c_1=0$ and $c_2=0$;
\item $c_0\neq 0$, $c_1=0$ and $c_2=0$.
\end{list}
For cases (i), $y_3=c_2y_2 \Rightarrow f_3(x)=y_3^2=c_2^2y_2^2=c_2^2f_2(x)$. Again, by comparing coefficients we may have $c_2^2=1$. This implies that $y_3=\pm y_2$, which is not possible. Similarly, case (ii) is impossible as well. Cases (iii) and (iv) are trivially excluded, so the proof is completed.
\end{proof}

\bigskip
\begin{theorem}
\cite[Theorem C]{bib: zeta2} Let $\mathcal{C}$ be a (smooth, projective) algebraic curve defined over an arbitrary field.  Let $H_1,...H_k\leq\mathrm{Aut}(\mathcal{C})$ be (finite) subgroups with $H_i\cdot H_j=H_j\cdot H_i\quad i\neq j$, and let $g_{ij}$ denote the genus of the quotient curve $\mathcal{C}/H_i\cdot H_j$.  Then, for $n_1,...,n_k\in\mathbb{Z}$, the conditions
\[
\sum n_i n_j g_{ij}=0, \\
\sum_{j} n_j g_{ij}=0,\quad 1\leq i\leq n,
\]
are each equivalent to $\sum n_i \varepsilon_{n_i}=0$, where $\varepsilon\in\mathrm{End}^0(\mathcal{J}_\mathcal{C})$ are \textit{idempotents}, and hence both imply the isogeny relation
\[
\prod_{n_i>0} \mathcal{J}^{n_i}_{\mathcal{C}/H_i} \sim \prod_{n_i<0} \mathcal{J}^{|n_i|}_{\mathcal{C}/H_i}.
\]
In particular, if $g_{ij}=0$ for $2\leq i<j\leq k$ and if
\[
g_\mathcal{C}=g_{\mathcal{C}/H_2}+...+g_{\mathcal{C}/H_k}
\]
then we have (by taking $H_1=\{1\}$ above):
\[
\mathcal{J}_\mathcal{C} \sim \mathcal{J}_{\mathcal{C}/H_2} \times ... \times \mathcal{J}_{\mathcal{C}/H_k}.
\]
\end{theorem}

\bigskip
\quad By considering $\mathcal{C}=\chi$, and taking $H_i$ as the following subgroups of $\mathrm{Aut}(\chi)=\mathrm{Aut}_{\mathbb{F}_q(t)}(\mathbb{F}_q(\chi))$:
\[
\left\{
\begin{array}{ccc}
H_1&=&\{1\} \\
H_2&=&<\sigma_2,\sigma_3> \\
H_3&=&<\sigma_1,\sigma_3> \\
H_4&=&<\sigma_1,\sigma_2> \\
\end{array}
\right.
\]
where
\[
\sigma_1:
\left\{
\begin{array}{ccc}
y_1&\mapsto&-y_1 \\
y_2&\mapsto&y_2 \\
y_3&\mapsto&y_3 \\
\end{array}
\right.
\]
\[
\sigma_2:
\left\{
\begin{array}{ccc}
y_1&\mapsto&y_1 \\
y_2&\mapsto&-y_2 \\
y_3&\mapsto&y_3 \\
\end{array}
\right.
\]
\[
\sigma_3:
\left\{
\begin{array}{ccc}
y_1&\mapsto&y_1 \\
y_2&\mapsto&y_2 \\
y_3&\mapsto&-y_3 \\
\end{array}
\right.
,
\]
we can easily check that the condition $H_i\cdot H_j=H_j\cdot H_i$ $\forall i,j=1,2,3,4$ is satisfied. Since $\mathbb{F}_q(\chi/H_2)=\mathbb{F}_q(\chi)_{H_2}=\mathbb{F}_q(t)(y_1)=\mathbb{F}_q(\mathcal{E}_1)$, we have $\chi/H_2\cong\mathcal{E}_1$. Similarly, we can also derive $\chi/H_3\cong\mathcal{E}_2$ and $\chi/H_4\cong\mathcal{E}_3$, and thus we know $g_{12}=g_{13}=g_{14}=1$. Moreover, as $\mathbb{F}_q(\chi/H_i\cdot H_j)=\mathbb{F}_q(\chi)_{H_i\cdot H_j}=\mathbb{F}_q(t)=\mathbb{F}_q(\mathbb{P}^1)$ for $2\leq i<j\leq 4$, $\chi/{H_i\cdot H_j}\cong \mathbb{P}^1$ $(2\leq i<j\leq 4)$, which implies $g_{ij}=0$ for $(2\leq i<j\leq 4)$. Therefore, $g_\chi=g_{\chi/H_2}+g_{\chi/H_3}+g_{\chi/H_4}$. Then using the particular case in the Theorem 3.3, we have
\begin{eqnarray*}
\mathcal{J}_\chi&\sim&\mathcal{J}_{\chi/H_2}\times\mathcal{J}_{\chi/H_3}\times\mathcal{J}_{\chi/H_4} \\
&=&\mathcal{J}_{\mathcal{E}_1}\times\mathcal{J}_{\mathcal{E}_2}\times\mathcal{J}_{\mathcal{E}_3}.
\end{eqnarray*}
This completes the proof of Proposition 3.1.

\bigskip
\begin{corollary}
Let $p$ be odd. Let $\mathcal{E}_i/\mathbb{F}_q$ $(i=1,2,3)$ be pairwise non-isomorphic elliptic curves. Suppose $j_3\in \mathbb{F}_q$. Then there exists a curve $\chi/\mathbb{F}_q$of genus 3 such that the Jacobian $\mathcal{J}_\chi$ is isogenous to $\mathcal{E}_1\times \mathcal{E}_2\times \mathcal{E}_3$ if one of the following is satisfied:
\newcounter{count6}
\begin{list}{(\roman{count6})}{\usecounter{count6}}
\item $|\mathcal{E}_1(\mathbb{F}_q)\cap\mathcal{E}_1[2]|=|\mathcal{E}_2(\mathbb{F}_q)\cap\mathcal{E}_2[2]|=|\mathcal{E}_3(\mathbb{F}_q)\cap\mathcal{E}_3[2]|=2$ and $\displaystyle\frac{\lambda_2(1-\lambda_1)}{\lambda_2-\lambda_1}$ is Legendre equivalent to $\lambda_3$;
\item $|\mathcal{E}_1(\mathbb{F}_q)\cap\mathcal{E}_1[2]|=|\mathcal{E}_2(\mathbb{F}_q)\cap\mathcal{E}_2[2]|=|\mathcal{E}_3(\mathbb{F}_q)\cap\mathcal{E}_3[2]|=4$ and $\displaystyle\frac{\lambda_2}{\lambda_1}$ is Legendre equivalent to $\lambda_3$.
\end{list}
\end{corollary}

\begin{proof}
This is the direct consequence of Lemma 2.7 and Proposition 3.1.
\end{proof}

\section{Main Result}

\quad Before we go through the main result, we should first know some facts. If there are three distinct integers $t_i$ $(i=1,2,3)$ satisfying one of the conditions in Lemma 2.1, then there exist three pairwise non-isomorphic elliptic curves $\mathcal{E}_i/\mathbb{F}_q$ $(i=1,2,3)$ such that $|\mathcal{E}_i(\mathbb{F}_q)|=q+1+t_i$. $\mathcal{E}_i$ is called \textit{an elliptic curve corresponding to $t_i$} in this paper.

\bigskip
\begin{definition}
Let $t_i$ $(i=1,2,3)$ be three distinct integers satisfying one of the conditions in Lemma 2.1. If there exists $\mathcal{E}_i/\mathbb{F}_q$ $(i=1,2,3)$ which are elliptic curves corresponding to $t_i$ with Legendre coefficients $\lambda_i$ such that:
\newcounter{count7}
\begin{list}{(\roman{count7})}{\usecounter{count7}}
\item $\displaystyle\frac{\lambda_{i_2}(1-\lambda_{i_1})}{\lambda_{i_2}-\lambda_{i_1}}$ is Legendre equivalent to $\lambda_{i_3}$, where $\{i_1,i_2,i_3\}=\{1,2,3\}$, and $j_{i_3}\in \mathbb{F}_q$; Then $t_i$ $(i=1,2,3)$ are said to be \textit{weakly Legendre consistent} and $\mathcal{E}_i$ are called \textit{elliptic curves corresponding to $t_i$ weakly}. Or else
\item $\displaystyle\frac{\lambda_{i_2}}{\lambda_{i_1}}$ is Legendre equivalent to $\lambda_{i_3}$, where $\{i_1,i_2,i_3\}=\{1,2,3\}$, $j_{i_3}\in \mathbb{F}_q$ and no $\mathcal{E}_i$ has the property $|\mathcal{E}_i(\mathbb{F}_q)\cap \mathcal{E}_i[2]|=2$; Then $t_i$ $(i=1,2,3)$ are said to be \textit{strongly Legendre consistent} and $\mathcal{E}_i$ are called \textit{elliptic curves corresponding to $t_i$ strongly}.
\end{list}
\end{definition}

\bigskip
With this new definition, the main theorem in this paper can thus be formulated.

\bigskip
\begin{theorem}
Let $q=p^r$ for a prime $p>2$ and an integer $r\geq 1$. Let $t_i$ $(i=1,2,3)$ be three distinct integers satisfying one of the conditions in Lemma 2.1. Then the function $\prod_{i=1}^3(T^2+t_iT+q)$ is the characteristic polynomial of the Frobenius endomorphism of the Jacobian of a curve of genus 3 over $\mathbb{F}_q$ if one of the following conditions is satisfied:
\newcounter{count8}
\begin{list}{(\roman{count8})}{\usecounter{count8}}
\item $q+1+t_i\equiv 2 \pmod 4$ $\forall i=1,2,3$ and $t_i$ $(i=1,2,3)$ are weakly Legendre consistent;
\item $q+1+t_i\equiv 0 \pmod 4$ $\forall i=1,2,3$ and $t_i$ $(i=1,2,3)$ are strongly Legendre consistent.
\end{list}
\end{theorem}

\begin{proof}
In the case (i), let $\mathcal{E}_i/\mathbb{F}_q$ be the elliptic curves corresponding to $t_i$ weakly. Then it is trivial that $|\mathcal{E}_1(\mathbb{F}_q)\cap\mathcal{E}_1[2]|=|\mathcal{E}_2(\mathbb{F}_q)\cap\mathcal{E}_2[2]|=|\mathcal{E}_3(\mathbb{F}_q)\cap\mathcal{E}_3[2]|=2$ and by definiton (without loss of generality) $\displaystyle\frac{\lambda_2(1-\lambda_1)}{\lambda_2-\lambda_1}$ is Legendre equivalent to $\lambda_3$ and $j_3\in \mathbb{F}_q$. Therefore by Corollary 3.4, there exists a curve $\chi/\mathbb{F}_q$of genus 3 such that the Jacobian $\mathcal{J}_\chi$ is isogenous to $\mathcal{E}_1\times \mathcal{E}_2\times \mathcal{E}_3$. In other word, the characteristic polynomial of the Frobenius endomorphism of the Jacobian of this curve is $\prod_{i=1}^3(T^2+t_iT+q)$.

In the case (ii), let $\mathcal{E}_i/\mathbb{F}_q$ be the elliptic curves corresponding to $t_i$ strongly. Then similarly by definiton, $|\mathcal{E}_1(\mathbb{F}_q)\cap\mathcal{E}_1[2]|=|\mathcal{E}_2(\mathbb{F}_q)\cap\mathcal{E}_2[2]|=|\mathcal{E}_3(\mathbb{F}_q)\cap\mathcal{E}_3[2]|=4$, $\displaystyle\frac{\lambda_2}{\lambda_1}$ is Legendre equivalent to $\lambda_3$ and $j_3\in \mathbb{F}_q$. Thus by applying Corollary 3.4, we get the same result as in the previous case.
\end{proof}

\bigskip
\begin{remark}
By adopting similar method in \cite{bib: zeta1}, we are only able to do the case up to $g=3$. For the cases $g\geq 4$, the method is invalid since the Proposition 3.1 breaks down.  Strictly speaking, we cannot construct a curve of genus $g$ by just considering the corresponding (smooth) algebraic curve to the composite field of the function fields of $g$ pairwise non-isomorphic elliptic curves, and then applying the Hurwitz formula to prove it is of genus g.
\end{remark}

\section{Further Research}
\begin{enumerate}
\item Find out all of the zeta-functions of curves of genus 2 over finite fields, including those not being determined in the paper \cite{bib: zeta1}.
\item Find out the intrinsic reason(s) why curves of genus 4 or above cannot be constructed by similar method in the case $g=2$ or $3$.
\item Improve the method used in this paper such that the conditions used for the construction are weaker.
\end{enumerate}

\section{Acknowledgement}
\quad I hereby give my sincere appreciation to my mentor, Dr. Eric Wambach. He gave me a lot of support and valuable advice about both my project and my future study plan during the ten-week SURF program. I give my special thanks to my sponsor, the Caltech Alumni Association (Hong Kong Chapter), for letting me gain this meaningful research experience. I am truly grateful for the great support from Mr. Chris Lyons, who helped me to tackle many problems encountered in the project while I was in Caltech. I would like to thank Mr. Gary Tsang as well for teaching me how to use LaTeX and helping me deal with the problems in compiling this paper.  Lastly, I offer my heartfelt thanks to all people who have given me a hand in any aspect.


\begin{thebibliography}{99}
\bibitem{bib: zeta1} C. XING, \emph{Zeta-functions of curves of genus 2 over finite fields}, Journal of Algebra \textbf{308} (2007) 734-741.
\bibitem{bib: zeta2} E. KANI AND M. ROSEN, \emph{Idempotent relations and factors of Jacobians}, Mathematische Annalen \textbf{284} (1989) 307-327.
\bibitem{bib: zeta3} M. DEURING, \emph{Die Typen der Multiplikatorenringe elliptischer Functionenk$\ddot{o}$rper}, Abh. Math. Sem. Univ. Hamburg \textbf{14} (1941) 197-272.
\bibitem{bib: zeta4} J. P. SERRE, \emph{Rational points on curves over finite fields}, notes by F. GOUVEA of lectures at Harvard University, 1985.
\bibitem{bib: zeta5} H. G. R$\mathrm{\ddot{U}}$CK, \emph{Abelian surfaces and Jacobian varieties over finite fields}, Compos. Math. \textbf{76} (1990) 351-366.
\bibitem{bib: zeta6} J. H. SILVERMAN, \emph{The Arithmetic of Elliptic Curves}, Springer-Verlag, New York, 1986.
\bibitem{bib: zeta7} D. HUSEMLLER, \emph{Elliptic Curves (2nd ed.)}, Springer, New York, 2004.
\bibitem{bib: zeta8} H. STICHTENOTH, \emph{Algebraic Function Fields and Codes}, Springer, Berlin, 1993
\bibitem{bib: zeta9} W. C. WATERHOUSE, \emph{Abelian varieties over finite fields}, Ann. Sci. $\mathrm{\acute{E}}$cole Norm. Sup. \textbf{(4)} 2 (1963) 521-560.
\end{thebibliography}
\end{document}